\documentclass[12pt]{article}
\usepackage{amsmath, amssymb, latexsym, amsthm}
\usepackage{fullpage}
\usepackage{color}
\usepackage{tikz}
\usepackage{graphicx}
\usepackage{stmaryrd,amsbsy,enumerate}
\usepackage{hyperref,todonotes}
\usepackage{mathscinet}
\usetikzlibrary{calc}
\usepackage{cancel}
\usepackage[mathlines]{lineno}
\usepackage{url}
\usepackage{longtable}
\usepackage{cite}

\if10     
\usepackage[mathlines]{lineno}
\newcommand*\patchAmsMathEnvironmentForLineno[1]{%
  \expandafter\let\csname old#1\expandafter\endcsname\csname #1\endcsname
  \expandafter\let\csname oldend#1\expandafter\endcsname\csname end#1\endcsname
  \renewenvironment{#1}%
     {\linenomath\csname old#1\endcsname}%
     {\csname oldend#1\endcsname\endlinenomath}}%
\newcommand*\patchBothAmsMathEnvironmentsForLineno[1]{%
  \patchAmsMathEnvironmentForLineno{#1}%
  \patchAmsMathEnvironmentForLineno{#1*}}%
\AtBeginDocument{%
\patchBothAmsMathEnvironmentsForLineno{equation}%
\patchBothAmsMathEnvironmentsForLineno{align}%
\patchBothAmsMathEnvironmentsForLineno{flalign}%
\patchBothAmsMathEnvironmentsForLineno{alignat}%
\patchBothAmsMathEnvironmentsForLineno{gather}%
\patchBothAmsMathEnvironmentsForLineno{multline}%
}
\linenumbers
\fi

\newtheorem{theorem}{Theorem} 
\newtheorem{theorem*}{Theorem} 

\newtheorem{lemma}[theorem]{Lemma}
\newtheorem{conjecture}[theorem]{Conjecture}

\newtheorem{claim}[theorem]{Claim}

\theoremstyle{definition}

\theoremstyle{remark}

\newcounter{arxiv}

\begin{document}
\setcounter{arxiv}{0}


\title{Inducibility of directed paths}
\author{
Ilkyoo Choi\thanks{
Ilkyoo Choi was supported by Basic Science Research Program through the National Research Foundation of Korea (NRF) funded by the Ministry of Education (NRF-2018R1D1A1B07043049), and also by Hankuk University of Foreign Studies Research Fund.
Department of Mathematics, Hankuk University of Foreign Studies, Yongin-si, Gyeonggi-do, Republic of Korea.
E-mail: \texttt{ilkyoo@hufs.ac.kr}
}\and
Bernard Lidick\'{y}\thanks{Research of this author is supported in part by NSF grants DMS-1600390 and DMS-1855653.
Department of Mathematics, Iowa State University, Ames, IA, USA. 
E-mail: {\tt lidicky@iastate.edu}}
\and
Florian Pfender\thanks{Research of this author is supported in part by NSF grants DMS-1600483 and DMS-1855622.
Department of Mathematical and Statistical Sciences, University of Colorado Denver, Denver, CO, USA. 
E-mail: {\tt Florian.Pfender@ucdenver.edu}
} 
}

\maketitle

\begin{abstract}
A long standing open problem in extremal graph theory is to describe all graphs that maximize the number of induced copies of a path on four vertices.
The character of the problem changes in the setting of oriented graphs, and becomes more tractable.
Here we resolve this problem in the setting of oriented graphs
without transitive triangles. 
\end{abstract}

An oriented graph is a directed graph without $2$-cycles.
In this paper, both undirected graphs and oriented graphs are considered, and the following definitions apply to both classes. 
For a graph $G$, we use $|G|$ to denote the number of vertices of $G$. 
We use $P_n$ to denote the path on $n$ vertices. 
Given 
graphs $G$ and $H$, the \emph{density} of $H$ in $G$, denoted $d_H(G)$, is defined to be
\[
d_H(G) = \frac{\#\text{ of induced copies of } H \text{ in } G}{\binom{|V(G)|}{|V(H)|}}.
\]
Given a fixed graph $H$ and a family $\mathcal G$ of graphs, investigating the maximum or minimum value of $d_H(G)$ over all graphs $G\in \mathcal G$ is an important area of research in extremal graph theory.
This question was formulated by Pippenger and Golumbic~\cite{PippengerGolumbic1975},
where they define the \emph{(maximal) inducibility} of a given graph $H$, denoted $I(H)$, as
\[
I(H) = \lim_{n \to \infty} \max_{|G|=n} d_H(G).
\]
They initiated the study by considering the family of undirected graphs, and they proved that for a graph $H$, the value $\max_{|G|=n} d_H(G)$ is nondecreasing and the limit $I(H)$ always exists.
A natural line of research is to to refine the question by considering an (infinite) family $\mathcal G$ of graphs (instead of the family of all graphs), and define the \emph{(maximal) inducibility of $H$ in $\mathcal{G}$} as 
\[
I(H,\mathcal{G}) = \lim_{n \to \infty} \max_{|G|=n,  G \in \mathcal{G}} d_H(G),
\]
if the limit exists.

Given a graph $H$, natural candidate graphs for maximizing the number of induced copies of $H$ are the iterated balanced blow-ups of $H$:
Partition $n$ vertices into $|V(H)|$ classes of sizes $\lceil\frac{n}{|V(H)|}\rceil$ and $\lfloor\frac{n}{|V(H)|}\rfloor$, corresponding to the vertices of $H$. Add all possible edges between any two classes corresponding to an edge of $H$. Now iterate this process inside each class.
When $H$ has $k$ vertices, a simple calculation shows that a sequence of iterated balanced blow-ups of $H$ gives $I(H) \geq k!/(k^k - k)$.
In the original paper by Pippenger and Golumbic~\cite{PippengerGolumbic1975}, they conjectured that for cycles of length at least $5$, this bound is tight. This is still an open question for almost all values of $k$. A graph $H$ is called a {\em fractalizer} if the iterated balanced blow-ups of $H$ are the only graphs maximizing the number of induced copies of $H$ for every $n$. In particular, for each fractalizer the bound is tight.
Interestingly, Fox, Huang, and Lee~\cite{FoxHuangLeeM,FoxHuangLee} showed that almost all graphs are fractalizers by considering random graphs. 
A similar result was recently published by Yuster~\cite{Yuster2018}.

Recently, investigating the inducibility of small graphs received much attention, thanks to the flag algebra method invented by Razborov~\cite{Razborov:2007}.
With the notable exception of $P_4$, the inducibility of graphs on at most four vertices is well understood, see Even-Zoha and Linial~\cite{bib-evan15}.
For $P_4$, the known best lower bound on $I(P_4)$ is $1173/5824 \approx 0.2014$, provided by a construction from~\cite{bib-evan15}, and the best upper bound $0.204513$, obtained by Vaughan~\cite{flagmatic} using flag algebras.

Inducibility of 5-vertex graphs is also not completely resolved. 
Recently, by proving $I(C_5) = \frac{1}{26}$, Balogh et al.~\cite{BaloghC5} determined that the bound is tight for $C_5$.
Before this result, Hatami et al.~\cite{Hatami:2011} and independently Grzesik~\cite{Grzesik:2011} solved
the Erd\H{o}s pentagon problem, which asks for the 
value of $I(C_5,\mathcal{T})$, where $\mathcal{T}$ is the family of triangle-free graphs.
In \cite{1712.08869}, this last problem is resolved for graphs of all orders.
The main difference between the problems of determining $I(C_5,\mathcal{T})$ and $I(C_5)$ is the extremal construction. 
A balanced blow-up of $C_5$ is the extremal construction when considering triangle-free graphs, and an iterated balanced blow-up of $C_5$ is the extremal construction when there are no restrictions on the graphs under consideration. 
When determining $I(C_5,\mathcal{T})$, the flag algebra method gives the exact upper bound on $I(C_5,\mathcal{T})$.
On the other hand, proving a tight upper bound on $I(C_5)$ by merely using flag algebras appears out of reach, and stability methods are used to improve the bound from flag algebra.

In this paper, we consider inducibility of oriented graphs. 
Hladk\'y, Kr\'a\soft{l}, and Norin~\cite{HKN} announced that $I(\vec{P}_3) = \frac{2}{5}$ and the extremal
construction is an iterated blow-up of $\vec{C}_4$.
We conjecture that this generalizes to longer oriented paths, namely, the number of induced copies of $\vec{P}_k$ is maximized by an iterated blow-up of $\vec{C}_{k+1}$.

\begin{conjecture}\label{conj:iter}
The number of induced copies of $\vec{P}_k$ over all oriented graphs on  $n$ vertices is maximized
by an iterated balanced blow-up of $\vec{C}_{k+1}$. As a consequence, 
\[
I(\vec{P}_k) =
\frac{k!}{(k+1)^{k-1}-1}.
\]
\end{conjecture}

Note that Conjecture~\ref{conj:iter} states that the graph maximizing the number of induced copies of $\vec{P}_{k}$ is the same graph as the graph conjectured to maximize the number of induced copies of $\vec{C}_{k+1}$. 
The statement regarding $I(\vec{C}_5)$ is a consequence of a result by Balogh et al.~\cite{BaloghC5} on $I(C_5)$.
Note that Hu et al.~\cite{orientedC4} resolved $I(\vec{C}_4)$, where the extremal example is an iterated blow-up of $\vec{C}_4$. This last construction is not extremal in the undirected case.

Let $\vec{T}_3$ denote the transitive tournament on three vertices.
Similar to triangle-free graphs in the class of undirected graphs, $\vec{T}_3$-free oriented graphs do not include iterated blow-ups of small graphs.
Therefore, extremal graphs often have simpler structure.
In this vein, we attack Conjecture~\ref{conj:iter} first by considering the same inducibility parameter but for $\vec{T}_3$-free oriented graphs. 
We formulate the following conjecture.

\begin{conjecture}\label{conjP4T3}
The number of induced copies of $\vec{P}_k$ over all $\vec{T}_3$-free oriented graphs on $n$ vertices 
is maximized by a balanced blow-up of $\vec{C}_{k+1}$. As a consequence, 
\[
I(\vec{P}_k, \mathcal{\vec T}) = \frac{k!}{(k+1)^{k-1}},
\]
where $\mathcal{\vec T}$ is the family of $\vec{T}_3$-free oriented graphs. 
\end{conjecture}

\begin{figure}[ht]
\begin{center}
\begin{tabular}{cc}
\includegraphics{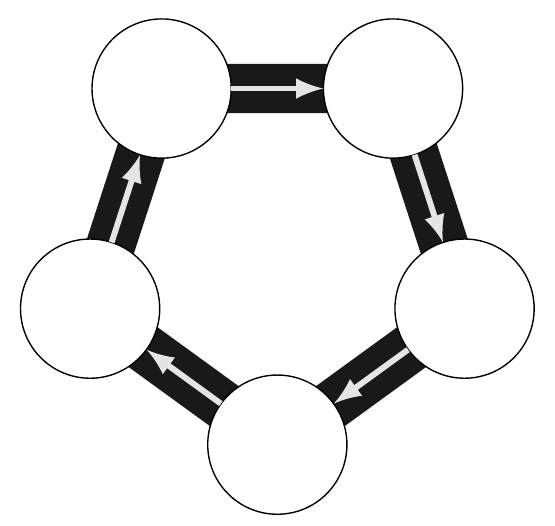} &
\includegraphics[page=2]{fig-construction} \\
(a) & (b)
\end{tabular}
\end{center}
\caption{(a) The blow-up of $\vec{C}_5$ and (b) the iterated blow-up of $\vec{C}_5$.}\label{fig-construction}
\end{figure}

In this paper, we prove Conjecture~\ref{conjP4T3} for $\vec{P}_4$, and we also show the uniqueness of the
extremal construction for sufficiently large graphs.

\begin{theorem}\label{thm-main}
Let $\mathcal{\vec T}$ be the family of oriented graphs without $\vec{T}_3$. 
\[
I(\vec{P}_4, \mathcal{\vec T}) = \frac{24}{125}.
\]
Moreover, for $n$ that is either sufficiently large or divisible by $5$,
the balanced blow-up of $\vec{C}_5$ is the only oriented $n$-vertex graph that maximizes the number of induced copies of $\vec{P}_4$ in $\mathcal{\vec T}$.
\end{theorem}

Our proof uses the flag algebra method. 
The method was developed by Razborov~\cite{Razborov:2007} and it has been successfully applied in various settings, see~\cite{bib-baber11,bib-baber14,bib-balogh17,bib-balogh14,bib-balogh15,bib-coregliano17,bib-das13,bib-falgas15,bib-gethner17,bib-glebov16,bib-goaoc15,bib-hladky17,bib-kral12,bib-kral13,bib-kral14,bib-lidicky17+,bib-cumm13}.
The method has been already described in many previous papers, so we do not describe
it here and merely use it as a black box. 
For an accessible introduction to flag algebras, see~\cite{1505.05200}.

In Section~\ref{sec:stab}, we first show the stability of the extremal construction, and then we obtain the exact result in Section~\ref{sec:exact}.
Utilizing a tool developed by Pippenger and Golumbic~\cite{PippengerGolumbic1975}  and Kr\'a\soft{l}, Norin, and Volec~\cite{KNV} in order to study the value of $I(C_k)$, we discuss upper bounds on $I(\vec{P}_k)$ and $I(\vec{P}_k, \mathcal{\vec T})$, where $\mathcal{\vec T}$ are  $\vec{T}_3$-free oriented graphs, for all $k$, in Section~\ref{longpaths}.

From now on, we will use $\vec{C}_5(n)$ to denote the balanced blow-up of $\vec{C}_5$ on $n$ vertices, see Figure~\ref{fig-construction}(a).

\newcommand{\dP}{d_{\vec{P}_4}}

\section{Stability}\label{sec:stab}

This section is devoted to proving the following stability lemma.
\begin{lemma}\label{lem:stability}
For every $\varepsilon > 0$, 
there exist $n_0$ and $\varepsilon' > 0$ such that 
every oriented $\vec{T}_3$-free graph $G$ of order $n \geq n_0$ with $\dP(G) \geq \frac{24}{125} - \varepsilon'$ is isomorphic to  $\vec{C}_5(n)$ after adding, removing, and/or reorienting at most $\varepsilon n^2$ edges.
\end{lemma}

Our main tools to prove Lemma~\ref{lem:stability} are flag algebras and a removal lemma. 
We use the following removal lemma, which follows from a more general theorem by Aroskar and Cummings~\cite{AroskarCummings}.

\begin{lemma}[Infinite Induced Oriented Graph Removal Lemma~\cite{AroskarCummings}]\label{lem:InfRem}
Let $\mathcal{F}$ be a (possibly infinite) set of oriented graphs. 
For every $\varepsilon_R > 0$, there exist $n_0$ and $\delta >0$ such that 
for every oriented graph $G$ of order $n\geq n_0$, 
if $G$ contains at most $\delta n^{v(H)}$ induced copies of $H$ for each $H$ in $\mathcal{F}$, then there exists $G'$ of order $n$ 
 such that $G'$ is induced $H$-free for all $H$ in $\mathcal{F}$ and $G'$ can be obtained from $G$ by adding/removing/reorienting at most $\varepsilon_R n^2$ edges.
\end{lemma}

The following derivation of Lemma~\ref{lem:InfRem}, as a special case of~\cite[Theorem~6]{AroskarCummings}, was provided by James Cummings. First start with a language 
$\mathcal{L}$ 
that has the equality symbol and one binary relation symbol $R$ that corresponds to a directed edge. The two axioms of the theory $T$ are 
$\forall x \neg R(x,x)$ and $\forall x \forall y (x \neq y \implies \neg (R(y,x) \wedge R(x,y)))$, representing no loops and no 2-cycles, respectively.
The models for $T$ now correspond exactly to oriented graphs.
(In \cite[Section 2.2]{AroskarCummings}, partitions that are potentially relevant are $\{\{1,2\}\}$ and $\{\{1\},\{2\}\}$.)
This gives $DH^R_{\{\{1,2\}\}}  = \{x : R(x,x)\} = \emptyset$ by the no 
loops axiom, and 
$DH^R_{\{\{1\},\{2\}\}} = \{(x, y) : x \neq y \text{ and }R(x, y)\}$. 
Hence the distance $d(G_1,G_2)$ between two oriented graphs $G_1$ and $G_2$ on a vertex set $V$ with $n$ elements is $D/n^2$, where $D$ is the number of ordered pairs $(v,w)$ of distinct elements of $V$ such that $G_1$ and $G_2$ disagree on  the existence of the directed edge from $v$ to $w$. Observe that a missing edge contributes $1$ to the distance while reversing an edge contributes $2$.
Finally, $p(G_1,G_2)$ is the usual density, which we denote as $d_{G_1}(G_2)$ in this paper.
In this setting, \cite[Theorem~6]{AroskarCummings} is Lemma~\ref{lem:InfRem}.

Let $\mathcal{F}$ be the oriented 4-vertex graphs depicted in Figure~\ref{fig:forb};
we call them the \emph{forbidden oriented graphs}.
A standard flag algebra calculation shows that the forbidden oriented graphs rarely appear in extremal examples.
 
\begin{lemma}\label{lem:FA}
For every $\delta > 0$, there exist $n_0$ and  $\varepsilon' > 0$ such
that every oriented $\vec{T}_3$-free graph $G$ of order $n \geq n_0$ with $\dP(G) \geq \frac{24}{125} - \varepsilon'$
contains at most $\delta n^{4}$ induced copies of an oriented graph in $\mathcal{F}$. Furthermore, $G$ contains at most $\delta n^3$ directed triangles.
\end{lemma}
\begin{proof}
We perform a calculation using the plain flag algebra framework. 
We obtain 
that if $(G_k)_{k \in \mathbb{N}}$ is a convergent sequence of oriented $\vec{T}_3$-free graphs, then
 ${\displaystyle\lim_{k \to \infty}}\dP(G_k) \leq \frac{24}{125}$.
Moreover, if  ${\displaystyle\lim_{k \to \infty}}\dP(G_k) = \frac{24}{125}$,
then for every $F \in \mathcal{F}$,  it follows that $\displaystyle\lim_{k \to \infty}d_F(G_k) = 0$.
It was sufficient to execute the calculation with flags on 4 vertices and two types.
Rounding was performed 
as described in~\cite{bib-balogh15}.
All technical details of the calculation, including rounded solution matrices, are available at \url{http://lidicky.name/pub/P4noT3} and on the \href{https://arxiv.org/abs/1811.03747}{arXiv}.
\end{proof}

\begin{proof}[Proof of Lemma~\ref{lem:stability}]
Fix $\varepsilon>0$. 
We use positive numbers $\delta$ and $\varepsilon_R$ that depend on Lemmas~\ref{lem:InfRem} and \ref{lem:FA}. 
We specify the dependency later at the end of the proof together with $n_0$ and $\varepsilon'$.

Let $G$ be an oriented graph of order $n \geq n_0$  with $\dP(G) \geq \frac{24}{125} - \varepsilon'$.
Notice that 
\begin{align}
\dP(G) \geq \frac{24}{125}-\varepsilon' = \dP(\vec{C}_5(n)) -\varepsilon' + o(1).\label{eq:C5}
\end{align}

By Lemma~\ref{lem:FA}, $G$ contains at most $\delta n^4$ induced copies of oriented graphs in $\mathcal{F}$, and at most $\delta n^3$ triangles.
By Lemma~\ref{lem:InfRem}, there exists an oriented $\vec{T}_3$-free graph $G'$ (on the same vertex set as $G$) differing from $G$ in at most $\varepsilon_R n^2$ pairs that avoids all oriented graphs in $\mathcal{F}$ and all triangles. 

\begin{figure}
\begin{center}
\includegraphics[page=3, scale=0.8]{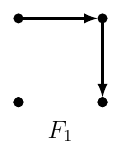}\hfill
\includegraphics[page=4, scale=0.8]{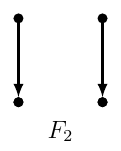}\hfill
\includegraphics[page=5, scale=0.8]{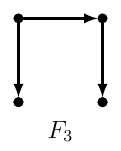}\hfill
\includegraphics[page=6, scale=0.8]{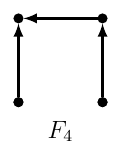}\hfill
\includegraphics[page=7, scale=0.8]{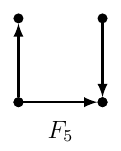}\hfill
\includegraphics[page=8, scale=0.8]{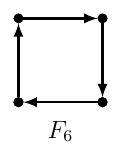}\hfill
\includegraphics[page=9, scale=0.8]{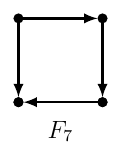}
\end{center}
\caption{Forbidden induced subgraphs.}\label{fig:forb}
\end{figure}

\begin{claim}\label{clm:blowupc5}
$G'$ is a (not necessarily balanced) blow-up of $\vec{C_5}$.
\end{claim}
\begin{proof}
Let $P=v_1,v_2,v_3,v_4$ form an induced $\vec{P}_4$ in $G'$.
We call a vertex $x$ a \emph{clone} of $v_i$ if $x$ and $v_i$ have exactly the same in-neighbors and out-neighbors on  $P$.
Let $X_1,\ldots,X_5$ be pairwise disjoint subsets of $V(G')$, where
\begin{align*}
X_i &= \{x : x \text{ is a clone of } v_i \text{ in } P\} & \text{ for } i\in\{1,2,3,4\},\\
X_5 &= \{ x: N^+(x)\cap V(P)=\{v_1\}, N^-(x)\cap V(P)=\{v_4\}   \}.
\end{align*}

Using the list $\mathcal{F}$ of forbidden oriented graphs, we show that $X_1,\ldots,X_5$ is a partition of $V(G')$.
Let $y \in V(G') \setminus \{v_1,v_2,v_3,v_4\}$.
By $F_1$, $y$ has at least one  neighbor on $P$. 
By $\vec{T}_3$ and $\vec{C}_3$, $y$ cannot have two consecutive neighbors on $P$.
In particular, $y$ cannot have three or four neighbors on $P$. 

Assume $y$ has exactly two neighbors $s$ and $t$ on $P$.
If $\{s, t\}=\{v_1, v_3\}$, then $y\in X_2$ by $F_6$ and $F_7$. 
If $\{s, t\}=\{v_2, v_4\}$, then $y\in X_3$ by $F_6$ and $F_7$. 
If $\{s, t\}=\{v_1, v_4\}$, then $y\in X_5$ by $F_3$ and $F_4$. 

Assume $y$ has exactly one neighbor $z$ on $P$. 
By $F_3$ and $F_4$, we know $z\not\in\{v_1, v_4\}$.
If $z=v_2$, then $y\in X_1$ by $F_3$, and if $z=v_3$, then $y\in X_4$ by $F_4$. 
\end{proof}

Next, we show $G'$ is close to being a balanced blow-up of $\vec{C}_5$.

\begin{claim}\label{cl:bal}
For every $\varepsilon_B > 0$, there exist $\varepsilon_R > 0$ and $\varepsilon' > 0$
such that if $n_0$ is sufficiently large and  $G$ and $G'$ differ in at most $\varepsilon_Rn^2$ pairs and $\dP(G) \geq \frac{24}{125}-\varepsilon'$,
then  $\vec{C}_5(n)$ and $G'$ differ in at most $\varepsilon_B n^2$ pairs.
\end{claim}
\begin{proof}
Given an oriented graph $H$, let $\vec{P_4}(H)$ denote the number of induced copies of $\vec{P_4}$ in $H$. 
Since $G'$ was obtained from $G$ by changing at most $\varepsilon_R n^2$ pairs, $\vec{P_4}(G')$ is large:
\begin{align}
\vec{P_4}(G')  &\geq \vec{P_4}(G) - \varepsilon_R n^4. \label{eq:a}
\end{align}

Notice that $\dP(G) \geq \frac{24}{125}-\varepsilon'$ implies that for sufficiently large $n$,
\begin{align}
 \vec{P}_4(G) \geq \left(\frac{24}{125}-\varepsilon'\right)\binom{n}{4} \geq  5\cdot \left(\frac{n}{5}\right)^4 - \varepsilon' n^4. \label{eq:b}
\end{align}
Recall that $G'$ is a (not necessarily balanced) blow-up of $\vec{C_5}$ by Claim~\ref{clm:blowupc5}.
By evaluating $\vec{P}_4(G')$ and combining it with \eqref{eq:a} and \eqref{eq:b} we obtain:
\begin{align*}
\left( \prod_i |X_i|\right)\cdot \left( \sum_i\frac{1}{|X_i|}  \right) \geq
 \vec{P}_4(G')  \geq
 5\left(\frac{n}{5}\right)^4 -  \varepsilon' n^4 - \varepsilon_R n^4 
=  \left(\frac{1}{5^3}- \varepsilon'-\varepsilon_R\right)n^4.
\end{align*}
The product on the left is maximized when $|X_i|={n\over 5}$ for each $i\in\{1,2,3,4,5\}$, and the maximum value is ${n^4\over 5^3}$.
Hence, for every $\varepsilon_B > 0$, there exist $\varepsilon_R > 0$ and $\varepsilon' > 0$ such that
if $\varepsilon' + \varepsilon_R$ is small enough, then
$ \left(\frac{1-\varepsilon_B}{5}\right) n \leq |X_i| \leq \left(\frac{1+\varepsilon_B}{5}\right) n$.
Therefore,  in order to obtain $\vec{C}_5(n)$ from $G'$, we need to move at most $\varepsilon_Bn$ vertices between parts,
which means changing at most $\varepsilon_Bn^2$ pairs.
\end{proof}

Let $\varepsilon_B = \varepsilon/2$. Let $\varepsilon_R \leq \varepsilon/2$ be small enough such $\varepsilon_R$, $\varepsilon_B$, $\varepsilon'$, and $n_0$ satisfy Claim~\ref{cl:bal}. 
Let $\delta > 0$ be small enough to satisfy Lemma~\ref{lem:InfRem} with $\varepsilon_R$.
Finally, let $\varepsilon'$ and $n_0$ be small and big, respectively, enough also for Lemma~\ref{lem:FA} when applied with $\delta$.
These choices will guarantee that $G$ is different from $\vec{C}_5(n)$ in at most
$(\varepsilon_B+\varepsilon_R)n^2 \leq \varepsilon n^2$ pairs.
\end{proof}

\section{Exact Result}\label{sec:exact}

This section contains the proof of Theorem~\ref{thm-main}. 
The proof follows the following outline.
We start with an extremal example $G$ with order $n$ and use Lemma~\ref{lem:stability} to conclude that $G$ is almost $\vec{C}_5(n)$. 
We first put ``unruly vertices'' aside and argue that the rest of $G$ is exactly a (not necessarily balanced) blow-up of $\vec{C}_5$. 
We then argue that the ``unruly vertices'' have drastically different sets of neighbors compared to the rest of the vertices in $G$.
Finally, we show that if there is a unruly vertex, then it would be in too few copies of an induced $\vec{P_4}$.
Hence, there are no ``unruly vertices'', and we finish the proof by showing that $G$ is a balanced blow-up of $\vec{C}_5$.

Given an oriented graph $H$ and a set of vertices $A\subset V(H)$, let $\vec{P}_4(H,A)$ be the number of induced $\vec{P}_4$'s in $H$ containing all vertices in $A$. 
If $A = \{a\}$, then we simplify the notation and write $\vec{P}_4(H,a)$ instead of $\vec{P}_4(H,\{a\})$.

\begin{proof}[Proof of Theorem~\ref{thm-main}]

For simplicity, we fix $\varepsilon$ to be sufficiently small, say $0.0000005$.
Let $n_0$ be big enough to apply Lemma~\ref{lem:stability} with $\varepsilon$  
such that every extremal oriented graph $H$ of order at least $n_0$ satisfies $\dP(H) \geq \frac{24}{125} - \varepsilon$. 

Let $G$ be an extremal oriented graph of order $ n \geq n_0$.
By Lemma~\ref{lem:stability}, the vertices of $G$ can be partitioned into five parts
$X_1,\ldots,X_5$ with sizes as equal as possible such that the graph can be turned into $\vec{C_5}(n)$ after adding, deleting, and/or reorienting the edge between at most $\varepsilon n^2$ pairs of vertices.

Call a pair of vertices where the adjacency needs to be changed \emph{weird}.
Use $f$ to denote the number of weird pairs in  $G$, and we know that
\[
f \leq \varepsilon n^2.
\]

For a vertex $v \in V(G)$, let $f_p(v)$ denote the number of weird pairs containing $v$.
Move every vertex with $f_p(v) \geq 0.001n$ to a new set $X_0$, so we know the following two inequalities hold:
\begin{align*}
|X_0| &\leq 2f/0.001n = 0.001n\\
f_p(v) &\leq 0.001n \text{ for all } v \in X_1\cup\cdots\cup X_5.
\end{align*}

Let $x_{min}$ and $x_{max}$ be a lower bound and an upper bound, respectively, on the size of $X_i$  for all $i \in \{1,2,3,4,5\}$.
Since we started with a balanced partition and $|X_0| \leq 0.001n$, we may use
\[
x_{min} = 0.198n \leq |X_i| \leq \lceil 0.2n\rceil  = x_{max} \hskip 1em\text{ for } i \in \{1,2,3,4,5\}.
\]

\begin{claim}\label{cl:nofunky}
$G - X_0$ is a blow-up of $\vec{C_5}$.
\end{claim}
\begin{proof}

Let $uv$ be a weird pair in $G - X_0$.
Obtain $G_{uv}$ from $G$ by making $uv$ not weird.
All induced $\vec{P}_4$'s in exactly one of $G_{uv}$ and $G$ must contain both $u$ and $v$.
Recall that $\vec{P}_4(H,\{u,v\})$ is the number of induced $\vec{P}_4$'s in $H$ containing vertices $u$ and $v$.

Note that $u$ and $v$ form, up to indexes, one of the Type 1--4 weird edges depicted in Figure~\ref{fig:funky}.
Type 5 is excluded, as otherwise $u$ and $v$ would have a common outneighbor forming a $\vec{T}_3$, which is forbidden.
Every $P_4$ contains two vertices other than $u$ and $v$ from two different sets of $X_1,\ldots,X_5$.
This would give $3\cdot x_{min}^2$ choices, yet, weird edges may prevent some of these from actually becoming $\vec{P}_4$s. 
Hence we get an easy lower bound.
\[
\vec{P}_4(G_{uv},\{u,v\}) \geq 3\cdot x_{min}^2 - n\cdot f_p(u) - n\cdot f_p(v)  - f.
\]
 
For $G$, note that every induced $\vec{P}_4$ containing the weird pair $uv$ must contain
either a vertex from $X_0$ or another weird pair. 
Hence
\[
\vec{P}_4(G,\{u,v\}) \leq n\cdot f_p(u) + n\cdot f_p(v)+ n\cdot |X_0| + f.
\]
By the extremality of $G$, we get
 \begin{align*}
\vec{P}_4(G,\{u,v\}) &\geq \vec{P}_4(G_{uv},\{u,v\}) \\
                         0 &\geq 3\cdot x_{min}^2 - 2n\cdot f_p(u) - 2n\cdot f_p(v) - n\cdot |X_0| - 2f \\
                         0 & \geq \left(3\cdot 0.198^2 - 5\cdot 0.001 - 0.000001\right)n^2 > 0.1 n^2,
\end{align*}
which is a contradiction.
\begin{figure}
\begin{center}
\includegraphics[page=1]{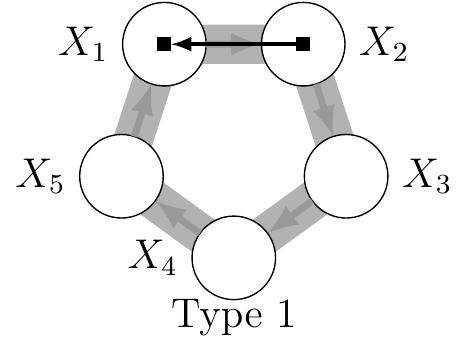}
\includegraphics[page=2]{fig-funky}
\includegraphics[page=3]{fig-funky}
\includegraphics[page=4]{fig-funky}
\includegraphics[page=5]{fig-funky}
\end{center}
\caption{Possible weird pairs.}\label{fig:funky}
\end{figure}
\end{proof}

\begin{claim}\label{cl:x0funky}
For $i\in\{1,2,3,4,5\}$ and a vertex $x \in X_0$, if $x$ is moved to $X_i$, then  $f_p(x) \geq 0.17n$.
\end{claim}
\begin{proof}
Fix a vertex $x\in X_0$ and without loss of generality let $i=1$.
There are five possible types of weird pairs containing $x$, see Figure~\ref{fig:funky}.

If $x$ is in a weird pair of Type 3 and Type 5, then $x$ is not adjacent to a vertex in $X_2$ and $X_5$, respectively,
because $\vec{T}_3$ is forbidden.
Hence, the number of weird edges incident with  $x$ is at least $x_{min} = 0.198n$.

Therefore, we may assume that all weird pairs involving $x$ are of Type 1, 2, or 4.
Let $xy$ be a weird pair of Type 1, 2, or 4.
Obtain $G_{xy}$ from $G$ by making $xy$ not weird.
Recall that Claim~\ref{cl:nofunky} implies that all weird pairs in $G-(X_0-x)$ contain $x$.
This allows us to use slightly better estimates when counting the possible induced $\vec{P}_4$'s containing $x$ and $y$.
We get
\[
\vec{P}_4(G,\{x,y\}) \leq  x_{max} f_p(x) + \binom{f_p(x)}{2} + |X_0|\cdot n
\]
and
\[
\vec{P}_4(G_{xy},\{x,y\}) \geq  3 x_{min}^2 - 2x_{max}f_p(x).
\]
By the extremality of $G$, we get
\begin{align*}
\vec{P}_4(G,\{x,y\}) &\geq \vec{P}_4(G_{xy},\{x,y\})\\
x_{max} f_p(x) + \binom{f_p(x)}{2} + |X_0|\cdot n &\geq  3x_{min}^2 - 2x_{max}f_p(x)\\
\frac{f_p(x)^2}{2} + 3x_{max}f_p(x) -  3x_{min}^2 + |X_0|\cdot n &\geq 0
\end{align*}
This gives $f_p(x) \geq 0.17n$ and finishes the proof of Claim~\ref{cl:x0funky}.
\end{proof}

\begin{claim}\label{cl:many}
For every vertex $x$, $\vec{P}_4(G,x) \geq 0.19\binom{n}{3}$.
\end{claim}
\begin{proof}
First we show that all vertices of $G$ are in approximately the same number of induced $\vec{P}_4$'s.
Suppose to the contrary that $x$ and $y$ are two vertices such that $\vec{P}_4(G,x) - \vec{P}_4(G,y) >  n^2$.
Obtain $G'$ from $G$ by deleting $y$ and adding a clone of $x$, denoted by $x'$. 
Note that there is no edge between $x$ and $x'$ in $G'$.
If $x'$ was in a $\vec{T}_3$, then $x$ would also be in a $\vec{T}_3$ since $x$ and $x'$ are not adjacent to each other. 
Hence $G'$ is $\vec{T}_3$-free.
The only induced $\vec{P_4}$'s that are different in $G$ and $G'$ are the ones containing both $y$ and $x'$.
Hence
\[
\vec{P}_4(G') - \vec{P}_4(G) = \vec{P}_4(G,x) - \vec{P}_4(G,y) - \vec{P}_4(G,\{x,y\})  > 0,
\]
since $ \vec{P}_4(G,\{x,y\}) < n^2$. This contradicts that $\vec{P}_4(G)$ is maximum.
Hence for two arbitrary vertices $x$ and $y$, $|\vec{P}_4(G,x) - \vec{P}_4(G,y)| \leq  n^2$.
Since $\vec{P}_4(G) = 0.192\binom{n}{4} + o(n^4)$ and 
every $\vec{P_4}$ contains four vertices, the average number of $\vec{P_4}$'s containing one fixed vertex is $0.192\binom{n-1}{3} + o(n^3) = 0.192\binom{n}{3} + o(n^3)$.
Therefore, $\vec{P}_4(G,x) \geq 0.19\binom{n}{3}$ for every vertex $x$ when $n$ is sufficiently large.
\end{proof}

\begin{claim}
$|X_0| = 0$.
\end{claim}
\begin{proof}
Suppose to the contrary that $x$ is a vertex in $X_0$.
We will show that $x$  violates Claim~\ref{cl:many}.

For $j \in [5]$, let $i_j$, $o_j$, and $n_j$ denote the number
of in-neighbors, out-neighbors, and non-neighbors, respectively, of $x$ in $X_j$ divided by $n$.
This allows us to count the number of induced $\vec{P}_4$'s containing $x$ and no other vertex from $X_0$.
To simplify the notation,  for all $j > 5$ we define $i_j=i_{j-5}$, $o_j=o_{j-5}$, and  $n_j=n_{j-5}$.
The following program provides an upper bound on the number of induced $\vec{P}_4$'s containing $x$ divided by $n^3$.
\def\XzerofunkyDeg{0.17}
\def\Xmax{0.21}
\def\Xmin{0.19}
\def\Xzeromax{0.001}

\[
(P)  \begin{cases}
\text{maximize}  
& \sum_{j=1}^5 \left(o_jn_{j+1}n_{j+2}+i_jo_{j+2}n_{j+3}+n_ji_{j+1}o_{j+3}+n_jn_{j+1}i_{j+3}\right) \\
\text{subject to}
&   i_j+o_j+n_j \le 0.21  \text{ for } j \in [5], \\
&   o_j+i_j+n_{j+1}+i_{j+1}+i_{j+2}+o_{j+2}+i_{j+3}+o_{j+3}+n_{j+4}+o_{j+4} \ge \XzerofunkyDeg  \text{ for } j \in [5],\\
& i_j,o_j,n_j\geq 0   \text{ for } j \in [5]. 
\end{cases}
\]

The objective in $(P)$ counts the number of induced $\vec{P}_4$'s containing $x$. 
The first set of constraints count relations between $x$ and vertices in each $X_j$.
Notice that we used a very generous upper bound on $|X_i|$.
The second set of constraints comes from Claim~\ref{cl:x0funky}, where we count the number of weird pairs containing $x$ if $x$ was in $X_j$.

We aim to provide an upper bound on the value of an optimal solution of $(P)$.
We do this by sampling points in the space of feasible solutions of $(P)$
and then upper bounding the maximum by using first derivatives.
Unfortunately, the program has ten variables, which seems to be too many for
generating a sufficiently refined grid.

Fortunately, the presence of some edges incident with $x$ blocks presence of other edges. 
If there are no edges from $x$ to $\bigcup_{i \in [5]} X_i$, then we can reverse all edges of $G$.
If there are still no edges from $x$ to $\bigcup_{i \in [5]} X_i$, then all neighbors of $x$ are  in $X_0$
and $x$ is in at most $|X_0|\binom{n}{2} < 0.001n^3$ induced $\vec{P}_4$'s, which contradicts Claim~\ref{cl:many}.

By symmetry, assume there is an edge directed from $x$ to a vertex in $X_1$.
This already prevents all edges between $x$ and $X_5$, and also edges from $x$ to $X_2$, since $G$ is $\vec{T}_3$-free.

All the possible combinations of allowed edges are depicted in Figure~\ref{fig-final_computation}, possible to verify by case analysis. 
In each of them, there are only four variables. 
We examine them separately, run a mesh optimization program, and show in the following
paragraphs that $v$ is in at most $0.08 \binom{n}{3}$  induced $\vec{P_4}$'s, which contradicts Claim~\ref{cl:many}.

The mesh optimization program works in the following way. 
For each variable, it samples 100 points uniformly distributed in $[0,0.21]$.
That means examining $100^4$ points.
For each of the points, we test if it is a feasible solution to $(P)$ and if yes,
then we remember the  solution with the highest value of the objective function of $(P)$.
The optimal solution of $(P)$ must be in each coordinate at distance at most $0.21/100$ from some
point we sampled\footnote{We actually also sample points that slightly violate the constraints of $(P)$ to make sure our grid captures the optimal solution if it is on the boundary of feasible solutions of $(P)$.}.
The largest value among the sampled feasible points is less than $0.04$.

The first partial derivative of the objective function in any variable is at most $6\cdot 0.21^2 = 2.52$.
Hence the difference between the point and the optimum is at most $4\cdot  2.52 \cdot \frac{0.21}{100} < 0.03$,
and the value of the optimum solution is at most $0.07$.

This implies that $x$ is in at most $0.07 \binom{n}{3}$  induced $\vec{P_4}$'s that avoid $X_0$.
There are at most $|X_0|\binom{n}{2}$ other  induced $\vec{P_4}$'s containing $x$. 
Hence there are at most $0.08 \binom{n}{3}$  induced $\vec{P_4}$'s containing $x$, which contradicts Claim~\ref{cl:many}.

\begin{figure}
\begin{center}
\includegraphics[page=1]{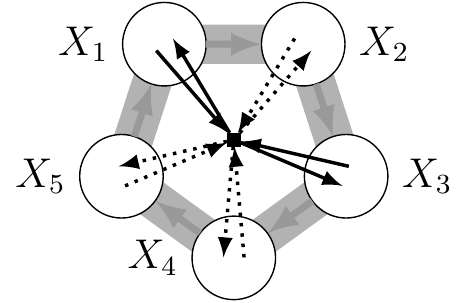}
\includegraphics[page=2]{fig-final_computation}\\[10pt]
\includegraphics[page=3]{fig-final_computation}
\includegraphics[page=4]{fig-final_computation}
\end{center}
\caption{Possible neighborhoods of a vertex $x$, depicted in the center, in $X_0$. If solid edges are present, then dashed edges are not present. The solid edges cover all options if there is an edge from $x$ to a vertex in $X_1$.}\label{fig-final_computation}
\end{figure}

\end{proof}

\begin{claim}
$G$ is a balanced blow-up. That is, $||X_i| - |X_j|| \leq 1$ for all $i, j\in[5]$.
\end{claim}
\begin{proof}
Suppose to the contrary that there exist $i$ and $j$ such that $|X_i| - |X_j| \geq 2$. Let $\{a,b,c\} = [5] \setminus \{i,j\}$.

We will obtain a contradiction by deleting a vertex in $X_i$ and duplicating a vertex in $X_j$, and show that this will increase the number  of induced $\vec{P_4}$'s.
Let $y \in X_i$ and $z \in X_j$.
Obtain $G'$ from $G$ by deleting $y$ and adding a clone of $z$, denoted by $z'$.
If $z'$ was in a $\vec{T}_3$, then $z$ would also be in a $\vec{T}_3$ as $z$ and $z'$ are not adjacent to each other.
Hence $G'$ is $\vec{T}_3$-free.
For $w \in [5]$, let $x_w=|X_w|$.
By checking all possible embeddings of an induced $\vec{P}_4$, we calculate
\begin{align*}
 \vec{P_4}(G,z) &= x_ax_bx_c +  x_ax_bx_i +  x_ax_ix_c +  x_ix_bx_c\\
 \vec{P_4}(G,y) &= x_ax_bx_c +  x_ax_bx_j +  x_ax_jx_c +  x_jx_bx_c\\
 \vec{P_4}(G,\{z,y\}) &= x_ax_b + x_ax_c + x_bx_c
\end{align*}
Notice that the induced $\vec{P_4}$'s containing only $z'$ or $y$ contribute to the difference of the number of induced $\vec{P_4}$'s in $G'$ and $G$. 
In particular,
\begin{align*}
\vec{P_4}(G') - \vec{P_4}(G) &= \vec{P_4}(G,z) - \vec{P_4}(G,y) - \vec{P_4}(G,\{z,y\})\\
&=   x_i (x_ax_b+  x_ax_c +  x_bx_c)   -  x_j(x_ax_b +  x_ax_c +  x_bx_c) - x_ax_b - x_ax_c -x_bx_c\\
&=   (x_i-x_j-1) (x_ax_b+  x_ax_c +  x_bx_c)\\
&\geq x_ax_b+  x_ax_c +  x_bx_c > 0,
\end{align*}
which contradicts that $\vec{P_4}(G)$ is maximum.
\end{proof}

The last remaining part of the proof of Theorem~\ref{thm-main} is to show that it holds for all $n$ divisible by $5$.
Assume that $n=5\ell$ for some $\ell \geq 1$ and $G$ is a graph maximizing the number of induced $\vec{P}_4$ among
all $n$-vertex graphs. 
Our goal is to show that $G$ is isomorphic to $\vec{C}_5(n)$, which is a balanced blow-up of $\vec{C}_5$ on $n$ vertices.
By the extremality of $G$, we get
\[
\vec{P}_4(G) \geq \vec{P}_4(\vec{C}_5(n)) = 5\ell^4.
\]
Now consider a blow-up $B$ of $G$, where every vertex of $G$ is replaced by $j$ vertices. 
That is, $B$ has $5j\ell$ vertices. 
Every $\vec{P}_4$ in $G$ yields $j^4$ copies of $\vec{P}_4$ in $B$. 
Hence $\vec{P}_4(B) = 5j^4\ell^4$.
If $j$ is sufficiently large, we have already proved, that $\vec{C}_5(5j\ell)$ is the unique extremal construction.
Hence
\[
5j^4\ell^4 = \vec{P}_4(\vec{C}_5(5j\ell)) \geq  \vec{P}_4(B) = 5j^4\ell^4.
\]
Therefore, $B$ is isomorphic to $\vec{C}_5(5j\ell)$. 
Since $B$ was obtained as a blow-up of $G$, we conclude that $G$ is isomorphic to $\vec{C}_5(n)$.
This finishes the proof of Theorem~\ref{thm-main}.
\end{proof}

\section{Longer directed paths}\label{longpaths}

We use methods developed for determining the inducibility of cycles in non-oriented graphs in order to obtain bounds for oriented paths of arbitrary length. 
The first general upper bound shown in Lemma~\ref{cl:PG} utilizes an approach by Pippenger and Golumbic~\cite{PippengerGolumbic1975}. 
Hefetz and Tyomkyn~\cite{HefetzTyomkyn} developed a more complicated approach, and Kr\'a\soft{l}, Norin, and Volec~\cite{KNV} recently improved the result via a simple counting argument. 
Similar technique was in~\cite{GrKi2019On} by Gresnik and Kielak. 
We use the method from \cite{KNV} for $\vec{T}_3$-free graphs, as proven in Lemma~\ref{lem:KNV}.

\begin{lemma}
\label{cl:PG}
\[
I(\vec{P}_k) \leq \frac{k!}{(k-1)^{k-1}}
\]
\end{lemma}
\begin{proof}
Let $G$ be an oriented graph on $n$ vertices.  
We try to build a path $v_1,\ldots,v_k$ by starting at $v_1$ and trying to append one vertex at a time.
We can choose $v_1$ to be any of the $n$ vertices.
Now in each step, let $w_i$ be the number of candidates for $v_i$.
That is, $w_1 = n$, $w_2 = |N^+(v_1)|$, $w_3 = |N^+(v_2) \setminus N(v_1)|$, and so on. 
Then, the total number of choices to build a path on $k$ vertices is
\[
\prod_{i=1}^k w_i =  n \cdot \prod_{i=2}^k w_i \leq n  \left(\frac{n}{k-1}\right)^{k-1} = \frac{n^{k}}{(k-1)^{k-1}}.
\]
Therefore,
\[
I(\vec{P}_k)  \leq \lim_{n \to \infty} \frac{\frac{n^{k}}{(k-1)^{k-1}}}{\binom{n}{k}} =  \frac{k!}{(k-1)^{k-1}}.
\]
\end{proof}

In the above proof of Lemma~\ref{cl:PG}, the worst case of the calculation is achieved when $w_i = \frac{n}{k-1}$ for all $i$. 
Instead of naively building the path, we will consider different orderings of the path (this is a trick inspired by~\cite{KNV}) in order to modify the worst case to be $w_i = \frac{n}{k}$ for all $i$.
This gives a further improvement on the bound, but it falls short of the best known construction, a blow-up of a $\vec{C}_{k+1}$. 

\begin{lemma}\label{lem:KNV}
\[
I(\vec{P}_k, \mathcal{\vec T}) \leq \frac{k!}{k^{k-1}},
\]
where $\mathcal{\vec T}$ is the family of $\vec{T_3}$-free oriented graphs. 
\end{lemma}
\begin{proof}
This proof follows the approach developed in \cite{KNV}.
Let $G$ be an oriented graph on $n$ vertices.
Let  $T=(z_1,\ldots,z_k)$ be a $k$-tuple of vertices of $G$.
We will consider $D_1(T), \ldots, D_k(T)$, where $D_i(T)$ denotes the following permuted $k$-tuple of $T$: 
\[
z_i,z_{i-1},\ldots, z_3,z_2,z_1,z_{i+1},z_{i+2},\ldots,z_k.
\]
Intuitively, we will think of the sequence  $z_1,\ldots,z_k$ as
an order of picking the vertices and $D_i(T)$ as an order in which these vertices form a copy of $\vec{P}_k$.
We define a weight $w$ as
\[
w(D_i(T)) = \prod_{j = 1}^k \frac{1}{n_{i,j}},
\]
where $n_{i,1} = n$ and $n_{i,j}$ is the number of possible candidates for $z_j$ given that $z_1,\ldots,z_{j-1}$ are already chosen 
and the copy of $\vec{P}_k$ is being built according to $D_i(T)$.

For a fixed $D_i$, we call a $k$-tuple $T$ \emph{good}, if the ordering of the vertices in $D_i(T)$ induces a copy of $\vec{P}_k$.
Let $D_i$ be fixed. 
By using reverse induction on $m$, the sum of the weights of all good $k$-tuples $(z_1,\ldots,z_k)$ with respect to $D_i$
that starts with $(z_1,\ldots,z_m)$ is at most $\prod_{j = 1}^m \frac{1}{n_{i,j}}$.
Hence, the total sum of weight of all good $k$-tuples with respect to $D_i$ is at most $1$.

By summing over all $i \in \{1, \ldots, k\}$, we conclude that
the sum of all weights of all $k$-tuples that are good for at least one $D_i$ is at most $k$.

Let $v_1,v_2,\ldots,v_k$ be an induced $\vec{P}_k$ in $G$.
For $i \in \{1, \ldots, k\}$, let 
\[
 T_i  = (v_i,v_{i-1},\ldots,v_2,v_1,v_{i+1},v_{i+2},\ldots,v_{k}).
 \]
 Notice that $T_i$ is a good $k$-tuple for $D_i$.
 We will later show that
 \begin{align}\label{eq:kk}
 \frac{k^k}{n^k} \leq w(D_1(T_1)) + w(D_2(T_2)) + \cdots + w(D_k(T_k)).
 \end{align}
Since the contribution to the sum of the weights of all good $k$-tuples is at least $\frac{k^k}{n^k}$ for each $\vec{P}_k$ in $G$, and the total sum is at most $k$, we conclude that the number of induced $\vec{P}_k$'s is at most $\frac{n^k}{k^{k-1}}$. 
By considering the limit, we get
\[
I(\vec{P}_k)  \leq \lim_{n \to \infty} \frac{\frac{n^{k}}{k^{k-1}}}{\binom{n}{k}} =  \frac{k!}{k^{k-1}}.
\]
It remains to prove~\eqref{eq:kk}. 
We will use the AM-GM inequality twice. 
The first use is
\begin{align}
\left(\prod_{i=1}^{k} w(D_i(T_i))  \right)^{\frac{1}{k}}  \leq \frac{w(D_1(T_1) + \cdots + w(D_k(T_k))}{k}. \label{eq:D}
\end{align}
The second comes in
\begin{align}
\left(\prod_{i=1}^{k} \frac{1}{w(D_i(T_i))}  \right)^{\frac{1}{k(k-1)}} 
&= \left(\prod_{i=1}^{k} n \cdot n_{i,2} \cdot n_{i,3} \cdots n_{i,k}   \right)^{\frac{1}{k(k-1)}}\nonumber \\
&=  n^{\frac{1}{k-1}} \cdot \left(\prod_{i=1}^{k} n_{i,2} \cdot n_{i,3} \cdots n_{i,k}   \right)^{\frac{1}{k(k-1)}} \nonumber \\
&\leq \frac{n^{\frac{1}{k-1}}}{k(k-1)} \cdot \left(\sum_{i=1}^{k} \left(n_{i,2} +  \cdots + n_{i,k}  \right) \right)  \label{eq:k}
\end{align}
Our next goal is to show that every vertex $x$ of $G$ can contribute at most one to $n_{i,2}+\cdots+n_{i,k}$ for each $i$ in \eqref{eq:k}, and moreover,
that there is one $i$, where $x$ does not contribute at all. This would give that the big sum in \eqref{eq:k} is upper bounded by $n(k-1)$.

If $x$ has no neighbors among $v_1,\ldots,v_k$, then it does not contribute at all.
Let $a$ be the smallest index such that $v_a$ and $x$ are adjacent. 
If $xv_a \in E(G)$, then $x$ does not contribute to $w(D_1(T_1))$. 
Let $b$ be the largest index such that $v_b$ and $x$ are adjacent. 
If $v_bx \in E(G)$, then $x$ does not contribute to $w(D_k(T_k))$.
Hence assume $v_ax \in E(G)$ and $xv_b \in E(G)$.
Since $G$ is $\vec{T}_3$-free, $x$ is not adjacent to $v_{a+1}$ and hence it does not contribute to $w(D_{a+1}(T_{a+1}))$.
Notice that if $G$ was not $\vec{T}_3$-free, then it might be the case that $b = a+1$ and $x$ would contribute to $w(D_{i+1}(T_{i+1}))$ for all $i \in \{1, \ldots, k\}$.

By replacing the big sum in \eqref{eq:k} by its upper bound $n(k-1)$ we obtain
\begin{align}
\left(\prod_{i=1}^{k} \frac{1}{w(D_i(T_i))}  \right)^{\frac{1}{k(k-1)}} &\leq \frac{n^{\frac{1}{k-1}}}{k(k-1)} \cdot n(k-1) = \frac{n^{\frac{k}{k-1}}}{k} \nonumber \\
\left(\prod_{i=1}^{k} \frac{1}{w(D_i(T_i))}  \right)^{\frac{1}{k}} &\leq  \frac{n^k}{k^{k-1}} \label{eq:i}
\end{align}
By combining the reciprocal  of \eqref{eq:i} with \eqref{eq:D}, which has been multiplied by $k$, we obtain
\[
 \frac{k^{k}}{n^k}  \leq  k\cdot \left(\prod_{i=1}^{k} {w(D_i(T_i))}  \right)^{\frac{1}{k}} \leq  
 w(D_1(T_1)) + w(D_2(T_2)) + \cdots + w(D_k(T_k)),
 \]
which proves \eqref{eq:kk} and finishes the proof of Lemma~\ref{lem:KNV}.
\end{proof}

\section{Conclusion}

Flag algebra calculations support Conjectures~\ref{conj:iter} and \ref{conjP4T3} for other small values of $k$.
For Conjecture~\ref{conj:iter}, we compute the following bounds:
\[
 0.1935483870 \approx \frac{6}{31} \leq I(\vec{P}_4) \leq 0.19356
\]
\[
 0.092664092  \approx \frac{24}{259} \leq I(\vec{P}_5) \leq 0.092676
\]
\[
0.0428418421  \approx \frac{120}{2801} \leq I(\vec{P}_6) \leq 0.04323  
\]

When restricted to $\vec T_3$-free oriented graphs, we get the following exact results for paths on five and six vertices, respectively: 
\begin{theorem}
For the family $\mathcal{\vec T}$ of $\vec{T}_3$-free oriented graphs, 
\begin{align*}
I(\vec{P}_5,\mathcal{\vec T}) &=  \frac{5!}{6^4}  & I(\vec{P}_6,\mathcal{\vec T}) &=  \frac{6!}{7^5}
\end{align*}
\end{theorem}
For $\vec P_7$, we compute a numerical upper bound matching Conjecture~\ref{conjP4T3} when $n \to \infty$. 
This means that we  successfully rounded numerical solution by flag algebras for $\vec P_5$ and $\vec P_6$, but fell short to do so for $\vec P_7$. 
We expect that the approach we used for $\vec{P}_4$ in this paper could also work for stability and exactness of $\vec{P}_5$ and $\vec{P}_6$.

\section*{Acknowledgements}

We would like to thank James Cummings for helping to  verify that Theorem~6 in~\cite{AroskarCummings} indeed implies Lemma~\ref{lem:InfRem}.

\bibliographystyle{abbrv}
\bibliography{refs.bib}

\end{document}